\documentclass{article}
\usepackage{amsmath,amssymb,amsthm,a4wide,graphicx,microtype, hyperref}
\newcounter{counter}
\newtheorem{lemma}[counter]{Lemma}
\newtheorem{theorem}[counter]{Theorem}
\newtheorem{remark}[counter]{Remark}
\newtheorem{corollary}[counter]{Corollary}
\newtheorem{definition}[counter]{Definition}

\newtheorem{proposition}[counter]{Proposition}
\newtheorem{innercustomthm}{Theorem}

\newenvironment{customthm}[1]
{\renewcommand\theinnercustomthm{#1}\innercustomthm}
{\endinnercustomthm}

\begin{document}

\title{A note on logarithmic equidistribution}
\author{Gerold Schefer}
\maketitle

\begin{abstract}For every algebraic number $\kappa$ on the unit circle which is not a root of unity we prove the existence of a strict sequence of algebraic numbers whose height tends to zero, such that the averages of the evaluation of $f_\kappa(z)=\log|z -\kappa|$ in the conjugates are essentially bounded from above by $-h(\kappa)$. This completes a characterisation on functions $f_\kappa$ initiated by Autissier and Baker-Masser, who cover the cases $\kappa=2$ and $|\kappa|\ne 1$ respectively. Using the same ideas we also prove analogues in the $p$-adic setting.\end{abstract}

\section{Introduction}
For $P\in\mathbb Z[X]\setminus \{0\}$ the logarithmic Mahler measure is defined by $m(P)=\int_0^1\log|P(e^{2\pi i t})|dt.$
Let $\overline{\mathbb Q}$ be an algebraic closure of the rational numbers, the field of algebraic numbers. Let $\alpha\in\overline{\mathbb Q}$. There is a unique irreducible polynomial $P=a_dX^d+\dots+a_0\in\mathbb Z[X]$ such that $P(\alpha)=0$ and $a_d\ge 1$. This is called the $\mathbb Z$-minimal polynomial of $\alpha$. The absolute logarithmic Weil height of $\alpha$ is defined by $h(\alpha)=\frac{m(P)}{d}$. There are also multiplicative versions of the Mahler measure and the height, denoted by $M(P)=e^{m(P)}$ and $H(\alpha)=e^{h(\alpha)}$. By Jensen's formula we have $M(P)=|a_d|\prod_{i:|\alpha_i|>1}|\alpha_i|$, where $\alpha_1,\dots,\alpha_d\in\mathbb C$ are the roots of $P$. This formula shows that the Mahler measure is multiplicative and satisfies $M(P)\ge 1$. Let us call a sequence $(\alpha_n)_{n\in\mathbb N}$ of algebraic numbers strict, if for every given $\alpha\in\overline{\mathbb Q}$ there are at most finitely many $n\in\mathbb N$ such that $\alpha_n=\alpha$. Let us write $\mu_{\infty}$ for the set of roots of unity in $\overline{\mathbb Q}$.

Bilu's Theorem 1.1 in \cite{Bilu} about algebraic numbers of small height states the following.
\begin{theorem}\label{thm:Bilu}Let $(\alpha_n)_{n\in\mathbb N}$ be a strict sequence of nonzero algebraic numbers such that \\
${\lim_{n\to\infty}h(\alpha_n)=0}$. Then for every bounded and continuous function $f:\mathbb C\setminus\{0\}\to \mathbb R$ we have \begin{equation}\label{eq:lim}\lim_{n\to\infty}\frac{1}{[\mathbb Q(\alpha_n):\mathbb Q]}\sum_{\sigma:\mathbb Q(\alpha_n)\to \mathbb C} f(\sigma(\alpha_n))= \int_0^1 f\left(e^{2\pi i t}\right) dt.\end{equation}\end{theorem}
It is a reasonable question to ask, whether we can relax the conditions on $f$ such that the limit \eqref{eq:lim} still holds. We consider a family of functions which have logarithmic singularities. For any algebraic number $\kappa$ we let $f_\kappa: \mathbb C\setminus \{\kappa\} \to\mathbb R$ be given by $f_\kappa(z)=\log|z-\kappa|$. We want to determine the set $\mathcal{K}_\infty$ of all $\kappa\in\overline{\mathbb Q}$ with the following property: For any strict sequence $(\alpha_n)_{n\in\mathbb N}$ of algebraic numbers such that $\alpha_n$ is never a conjugate to $\kappa$ and $\lim_{n\to\infty}h(\alpha_n)=0$ the limit \eqref{eq:lim} holds for the function $f=f_\kappa$. The integral equals $\log \max\{1,|\kappa|\}$ by Jensen's formula.

Autissier proved in \cite{Aut} that $2\not\in \mathcal{K}_\infty$ and thus, provided a counterexample to a conjecture of Pineiro-Szpiro-Tucker p. 241 \cite{PST}. Breuillard showed in Lemma 6.5 \cite{Br}, that $1$ belongs to $\mathcal{K}_\infty$. From this one can show that $\mu_\infty\subseteq \mathcal{K}_\infty$. For the sake of completeness we give a quantitative version in Corollary \ref{cor:muInf}. The fact $\mu_\infty\subseteq \mathcal{K}_\infty$ also follows from the work of Chambert-Loir and Thuillier, Théorème 1.2 b) \cite{CLT}.
Frey also looked at $\kappa=1$ and showed in Lemma 3.5 \cite{Fr} that if $[\mathbb Q(\alpha):\mathbb Q]\ge 16, 0<\delta<1/2$ and $h(\alpha)^{1/2}\le 1/2$, then $\frac{40}{\delta^4}h(\alpha)^{1/2-\delta}$ is an upper bound for the average in \eqref{eq:lim}. For $\kappa=0$ it is an exercise to show that the average is bounded in absolute value by the height, see (1.4) in \cite{BM}, and thus $0\in \mathcal{K}_\infty$. In the same paper, R. Baker and D. Masser were able to prove in Theorem 1.12 \cite{BM} that every $0\neq\kappa\in \mathcal{K}_\infty$ satisfies $|\kappa|=1$. More precisely, they prove:

\begin{theorem}\label{thm:BM}Let $\kappa\ne  0$ be algebraic with $|\kappa| \ne 1$. Then there
exists a strict sequence of algebraic numbers $(\alpha_n)_{n\in\mathbb N}$ such that no $\alpha_n$ is a conjugate of $\kappa$, $\lim_{n\to\infty}h(\alpha_n)= 0$ and $$\lim_{n\to \infty}\left|\frac{1}{[\mathbb Q(\alpha_n):\mathbb Q]}\sum_{\sigma:\mathbb Q(\alpha_n)\to \mathbb C} \log|\sigma(\alpha_n)-\kappa| - \int_0^1 \log|e^{2\pi i t}-\kappa| dt\right|= | \log |\kappa|| > 0.$$\end{theorem} 

In particular, the errors do not converge to zero and hence all the $\kappa$ considered in the theorem cannot belong to $\mathcal{K}_\infty$. Just below Theorem 1.12 \cite{BM}, R. Baker and D. Masser state that it would be interesting to know what happens for $|\kappa|=1$. Since the roots of unity are already understood, we are left with algebraic numbers $\kappa$ on the unit circle which are not roots of unity. Note that in this case we have $\int_0^1 f_\kappa(e^{2\pi i t}) dt = \log\max\{ 1, |\kappa|\} =0$. The main aim of this article is to show that such $\kappa$ never belong to $\mathcal{K}_\infty$. This fact is implied by the following theorem.

\begin{theorem}\label{thm:kappa}Let $\kappa\in\overline{\mathbb Q}\setminus\mu_\infty$ be on the unit circle and $d=[\mathbb Q(\kappa):\mathbb Q]$. Then there exists a strict sequence of algebraic numbers $(\alpha_n)_{n\in\mathbb N}$ such that no $\alpha_n$ is a conjugate of $\kappa$, $\lim_{n\to\infty}h(\alpha_n)= 0$,$$\limsup_{n\to\infty}\frac{1}{[\mathbb Q(\alpha_n):\mathbb Q]}\sum_{\sigma: \mathbb Q(\alpha_n)\to\mathbb C}\log|\sigma(\alpha_n)-\kappa|\le -h(\kappa)<0\text{ and }$$
$$\liminf_{n\to\infty}\frac{1}{[\mathbb Q(\alpha_n):\mathbb Q]}\sum_{\sigma: \mathbb Q(\alpha_n)\to\mathbb C}\log|\sigma(\alpha_n)-\kappa|\ge -d(h(\kappa)+\log(2)).$$\end{theorem}

Putting all the results together we find that $\mathcal{K}_\infty=\{0\}\cup\mu_{\infty}$.

\paragraph{The $p$-adic analogue.}
Let $p$ be a prime and $\overline{\mathbb Q}\subseteq \overline{\mathbb Q_p}$ be an algebraic closure of the rationals, $\alpha$ and $\kappa$ algebraic numbers which are not conjugated to each other. We want to understand for which $\kappa$ the following limit holds. \begin{equation}\label{eq:gen}\lim_{\substack{h(\alpha)\to 0\\ [\mathbb Q(\alpha):\mathbb Q]\to \infty}}\frac{1}{[\mathbb Q(\alpha):\mathbb Q]}\sum_{\sigma:\mathbb Q(\alpha)\to \overline {\mathbb Q_p}} \log |\sigma(\alpha)-\kappa|_p = \log\max\{1,|\kappa|_p\}.\end{equation}
 It turns out that Theorems \ref{thm:BM} and \ref{thm:kappa} also hold in this setting, if we replace the integral by $\log\max\{1,|\kappa|\}$.
\begin{customthm}{\ref{thm:BM}'}
\label{thm:BMgen}Let $0\ne \kappa\in\overline{\mathbb Q}\subseteq\overline{\mathbb Q_p}$ be such that $|\kappa|_p \ne 1$. Then there
exists a strict sequence of algebraic numbers $(\alpha_n)_{n\in\mathbb N}$ such that no $\alpha_n$ is a conjugate of $\kappa$, $\lim_{n\to\infty}h(\alpha_n)= 0$ and $$\lim_{n\to \infty}\left|\frac{1}{[\mathbb Q(\alpha_n):\mathbb Q]}\sum_{\sigma:\mathbb Q(\alpha_n)\to \overline{\mathbb Q_p}} \log|\sigma(\alpha_n)-\kappa|_p - \log\max\{1,|\kappa|_p\}\right|= | \log |\kappa|_p|> 0.$$\end{customthm}

\begin{customthm}{\ref{thm:kappa}'}\label{thm:kappaGen}Let $\kappa\in\overline{\mathbb Q}\setminus\mu_\infty$ such that $|\kappa|_p=1$ and $d=[\mathbb Q(\kappa):\mathbb Q]$. Then there exists a strict sequence of algebraic numbers $(\alpha_n)_{n\in\mathbb N}$ such that no $\alpha_n$ is a conjugate of $\kappa$, $\lim_{n\to\infty}h(\alpha_n)= 0$,$$\limsup_{n\to\infty}\frac{1}{[\mathbb Q(\alpha_n):\mathbb Q]}\sum_{\sigma: \mathbb Q(\alpha_n)\to\overline{\mathbb Q_p}}\log|\sigma(\alpha_n)-\kappa|_p\le -h(\kappa)<0\text{ and }$$
$$\liminf_{n\to\infty}\frac{1}{[\mathbb Q(\alpha_n):\mathbb Q]}\sum_{\sigma: \mathbb Q(\alpha_n)\to\overline{\mathbb Q_p}}\log|\sigma(\alpha_n)-\kappa|_p\ge -d(h(\kappa)+\log(2)).$$\end{customthm}
Let us define $\mathcal{K}_p = \{\kappa\in\overline{\mathbb Q}: \eqref{eq:gen}\text{ holds}\}$.
Then Theorems \ref{thm:BMgen} and \ref{thm:kappaGen} imply that $\mathcal K_p\subseteq \{0\}\cup\mu_\infty.$ Using that $1\in\mathcal K_\infty$ one can show that the last inclusion is an equality. For $\kappa=0$ one uses that for integers $k\in\mathbb Z\setminus \{0\}$ we have $1/|k|\le |k|_p\le 1$ to bound the average also in this case by the height of $\alpha$.
\begin{proposition}\label{pr:K}For all $\nu\in M_{\mathbb Q}$ we have $\mathcal{K}_\nu =\{0\}\cup\mu_\infty$.\end{proposition} 

\paragraph{Proof strategy for Theorem \ref{thm:kappaGen}.} Let $\nu\in M_{\mathbb Q},\alpha\in\overline{\mathbb Q}\subseteq\overline{\mathbb Q_\nu}$ and $T\in \mathbb Z[X]$ be its $\mathbb Z$-minimal polynomial with leading coefficient $t$ and degree $d$. Then $$\frac{1}{d}\sum_{\sigma: \mathbb Q(\alpha)\to\overline{\mathbb Q_\nu}}\log|\sigma(\alpha)-\kappa|_\nu=\frac{1}{d}\log \left(\prod_{\sigma:\mathbb Q(\alpha)\to\overline{\mathbb Q_\nu}}\left|\sigma(\alpha)-\kappa\right|_\nu\right)= \frac{1}{d}\log|T(\kappa)/t|_\nu.$$ The goal is therefore to construct a polynomial $T$ which is irreducible, takes a very small but nonzero value at $\kappa$ and such that its roots have small height.
The main idea is to use Minkowski's Linear Forms Theorem and its generalization of Mahler to construct a polynomial $A$ with controlled coefficients that has small but non-zero value at $\kappa^n$ for a parameter $n\ge 1$. Nonvanishing is achieved by bounding the Mahler measure of $A$ from above and below. For the upper bound we use the control on the coefficients of $A$. For the lower bound we use the height property $H(\kappa^n)=H(\kappa)^n$.  We deduce that if $A$ would vanish on $\kappa^n$, then its Mahler measure is bounded from below by a power of the multiplicative height of $\kappa$.

To get an irreducible polynomial, we modify $A$ by multiplying it by some fixed prime and adding some multiples of the minimal polynomial of $\kappa$ such that we can apply a generalized version of Eisenstein's criterion to ensure that there is an irreducible factor $T$ of large degree. To see that $T(\kappa)$ is still very small, we have to bound the value at $\kappa$ of the other factor from below. The proof uses the product formula.

\paragraph{Organisation of the article} In the next section we prove inequalities between the Mahler measure of $P$ and its $l_1$-norm, estimate from below the value of a polynomial at $\kappa$ which does not vanish and finally we give a generalisation of Eisenstein's irreducibility criterion which will allow us to find an irreducible factor of large degree.

The proof of Theorem \ref{thm:BMgen} is the same as the proof for Theorem 1.12 in \cite{BM}. Then we prove Theorem \ref{thm:kappaGen}, where for the construction of $A$ we handle the archimedian and the $p$-adic case separately. In the appendix we prove that the limit \eqref{eq:gen} holds if $\kappa$ is a root of unity and again distinguish the archimedian and the $p$-adic case. In the archimedian setting we use a quantitative version of Bilu's theorem by R. Baker and D. Masser.

\paragraph{Acknowledgements} I want to thank R. Baker and D. Masser for asking this nice question and Philipp Habegger for his advice and giving me the hint to use Schmidt's result to get a polynomial which takes a very small value at $\kappa$. This is part of my PhD thesis. Further, I gratefully acknowledge support from the
Swiss National Science Foundation grant “Diophantine Equations: Special Points,
Integrality, and Beyond” (n\textsuperscript{o}  200020\_184623).

\section{Preparation}

\begin{definition}We denote by $\mathbb N$ the natural numbers $\{1, 2, 3,\dots\}$ and by $|\cdot|:\mathbb C\to \mathbb R$ the complex absolute value. The $n$-th roots of unity are denoted by $\mu_n=\{z\in\overline{\mathbb Q}: z^n=1\}$ and $\mu_\infty=\bigcup_{n\ge 1}\mu_n$ is the set of all roots of unity. Here $\overline{\mathbb Q}$ is any algebraic closure of the rationals.

For a prime $p$ the $p$-adic numbers $\mathbb Q_p$ are the completion of $\mathbb Q$ with respect to the $p$-adic absolute value $|\cdot|_p$, which extends to $\mathbb Q_p$. The $p$-adic integers are given by $\mathbb Z_p=\{x\in\mathbb Q_p : |x|_p \le 1\}.$

We use the same symbol $\overline{\mathbb Q}$ to denote an algebraic closure of $\mathbb Q$ in $\overline{\mathbb Q_\nu}$ for some $\nu\in M_{\mathbb Q}$. If $\nu$ is not specified, $\overline{\mathbb Q}$ denotes the algebraic numbers in $\mathbb C$. 

We denote by $M_K^0$ the finite and by $M_K^\infty$ the infinite places of a number field $K$. 

Let $P=a_dX^d+\dots+a_0 \in \mathbb Z[X]\setminus\{0\}$ be a polynomial of degree $d$. Then we define the $l_1$-norm $|P|_1=|a_d|+\dots+|a_0|$ and $|0|_1=0$. \end{definition}

\begin{lemma}\label{lem:ineqMM}Let $P\in \mathbb Z[X]$ be of degree $d\ge 0$. Then we have $M(P)\le |P|_1 \le 2^{d}M(P).$\end{lemma}
\begin{proof}The first inequality is proved in Lemma 1.6.7 in \cite{BG}. In the proof of this lemma, it is shown that the coefficients $a_i$ of $P$ satisfy $$|a_{d-r}|\le \binom{d}{r} M(P)\text{ for all }r=0,\dots, d.$$ Summing up these inequalities we get $$|P|_1=\sum_{r=0}^d|a_{d-r}|\le M(P)\sum_{r=0}^d\binom{d}{r}=2^dM(P).\qedhere$$\end{proof}

\begin{lemma}\label{lem:Imge12}Let $u\in\mathbb C\setminus\{\pm 1\}$ be on the unit circle. Then there are infinitely many $n\in\mathbb N$ such that the imaginary part of $u^n$ is at least $1/2$.\end{lemma}
\begin{proof}Let $\arg(z)\in[0,2\pi)$ be the argument of $z$ and define the angle distance $$\mathrm{dist}(z_1, z_2)=\min\{|\arg(z_1)-\arg(z_2)|, |\arg(z_1)-\arg(z_2)\pm 2\pi|\},$$ a distance between two points $z_1, z_2$ on the unit circle. The condition $\Im(\kappa^n)\ge 1/2$ is equivalent to $\arg(\kappa^n)\in [\pi/3, 2\pi/3]$. Suppose that there are $n<m$ such that $0<\delta=\mathrm{dist}(u^n,u^m)\le\pi/3$, and let $N=\lceil 2\pi/\delta\rceil.$ Let $l=m-n$ and $k\in\mathbb N$. Then by rotation invariance of the angle distance function, $u^k, u^{k+l},\dots, u^{k+Nl}$ have the property that the angle between two consecutives is $\delta$ and since there are at least $2\pi/\delta$ of them, there is at least one of them in any sector of angle at least $\delta$. In particular, one of them has imaginary part at least $1/2$. Since $k$ was arbitrary, this shows that under the assumption we made, the conclusion of the lemma holds. 

Now consider $1, u, \dots, u^5$. If all are different, the smallest angle distance between two of them is at most $\pi/3$ and the assumption is satisfied. Thus, we are left with roots of unity of order bounded by five. So let $\zeta\in\mu_\infty$ be of order $a\in \{3,4,5\}$. Then there exists $b$ coprime to $a$ such that $\zeta=e^{2\pi i b/a}$. By Bézout's theorem, there exist $u,v\in\mathbb Z$ such that $au+bv=1$. Then for all $k\in\mathbb N$ we have $\zeta^{ak+v}= e^{2\pi i bv/a} = e^{2\pi i/a}$ whose imaginary part is at least $1/2$.\end{proof}

\begin{lemma}\label{lem:PatKappa}Let $\kappa\in \overline{\mathbb Q}, K=\mathbb Q(\kappa)$ and $d=[K:\mathbb Q]$. Then for every $P\in\mathbb Z[X]$ such that $P(\kappa)\ne 0$ and a place $\nu_0$ of $K$ we have $|P(\kappa)|_{\nu_0}\ge H(\kappa)^{-d\deg(P)}|P|_1^{-d}$.\end{lemma}

\begin{proof}Let $D=\deg(P), K=\mathbb Q(\kappa)$ and $M_K$ its set of places. For every place $\nu\in M_K$ let $d_\nu=[K_\nu:\mathbb Q_\nu]$, where $K_\nu$ is the completion with respect to $\nu$. If $\nu|p$ we normalize $|\cdot|_\nu$ such that $|p|_\nu = 1/p$. Proposition 1.4.4 in \cite{BG} states that the product formula holds: For all $x\in K\setminus\{0\}$ we have $$\prod_{\nu\in M_K}|x|_\nu^{d_\nu}=1.$$ 

For all infinite places we can bound $|P(\kappa)|_\nu\le |P|_1\max\{1,|\kappa)|_\nu\}^D$ and for every finite place we have $|P(\kappa)|_\nu\le \max\{1,|\kappa|_\nu\}^D$. For $\nu_0$ we will use a different bound. If $\nu_0\in M_K^\infty$ we can bound $|P(\kappa)|_{\nu_0}\le |P(\kappa)|_{\nu_0}|P|_1\max\{1,|\kappa|_{\nu_0}\}^{D}$. If $\nu_0\in M_K^0$ we have $|P(\kappa)|_{\nu_0}\le |P(\kappa)|_{\nu_0}\max\{1,|\kappa|_{\nu_0}\}^{D}$. Then by the product formula and since $P(\kappa)\ne 0$ we have $$1=\prod_{\nu\in M_K} |P(\kappa)|_\nu^{d_\nu}\le |P(\kappa)|_{\nu_0}^{d_{\nu_0}}|P|_1^{d}\prod_{\nu\in M_K}\max\{1,|\kappa|_\nu\}^{d_\nu D}=|P(\kappa)|_{\nu_0}^{d_{\nu_0}}|P|_1^{d}H(\kappa)^{dD}.$$ Now divide by $|P|_1^{d}H(\kappa)^{dD}$ and take the $d_{\nu_0}$-th root  to get $|P(\kappa)|_{\nu_0}\ge |P|_1^{-d/d_{\nu_0}}H(\kappa)^{-dD/d_{\nu_0}}.$ Since $|P|_1$ and $H(\kappa)$ are at least one, we can conclude.\end{proof}

Next we come to a generalization of the Eisenstein criterion. We obtain the classical theorem
if $e = \deg(R)$.

\begin{lemma}\label{lem:Eis}Let p be a prime number, let $e \ge 1$ be an integer, and suppose that $R, S \in
\mathbb Z[X]$ are polynomials such that $p^2 \nmid R(0)$ and their reductions in $\mathbb F_p[X]$ satisfy $\overline R\ne 0$ and $\overline R = X^e\overline S$. Then $R$ is divisible
by an irreducible polynomial in $\mathbb Z[X]$ of degree at least $e$.\end{lemma}

\begin{proof}We factor $R = P_1\cdots P_g$ with $P_1,\dots , P_g$ irreducible in $\mathbb Z[X]$. Then $\overline R =
X^e\overline S = \overline{P_1} \cdots \overline{P_g}$ is nonzero in $\mathbb{F}_p[X]$, the bar indicates reducing coefficients modulo
$p$. As $e \ge 1$ we have $X | \overline{P_1}$ after permuting the $P_i$.
Now suppose $\deg(\overline{P_1}) < e,$ then $X^e \nmid \overline{P_1}$. In this case, there exists $i \in \{2, . . . , g\}$
with $X | \overline{P_i}$. So $p | P_1(0)$ and $p| P_i(0)$, which implies $p^2 | P_1(0) \cdots P_g (0) = R(0)$
and thus contradicts the hypothesis on $R$.
So we have $\deg \overline{P_1} \ge e$ and in particular $\deg P_1 \ge e.$\end{proof}
This is a generalization of the ``Eisenstein criterion". We obtain the classical theorem if $e=\deg(R)$.
\begin{lemma}\label{lem:intLincomb} Let $x\in\overline{\mathbb Q_p}$ be such that $|x|_p\le1$ and let $d=[\mathbb Q_p(x):\mathbb Q_p]$. Then for every $n\in\mathbb N_0$ there exist $\lambda_0,\dots, \lambda_{d-1}\in\mathbb Z_p$ such that $x^n=\lambda_0+\dots+\lambda_{d-1}x^{d-1}$.\end{lemma}

\begin{proof}Note that if $d=1$, then $x^n=\lambda_0\in\mathbb Z_p$. So we may assume that $d\ge 2$.
If $n<d$ we can choose $\lambda_i=\delta_{n,i}$. Let us do the remaining cases by induction on $n$, starting with $n=d$. Let $X^d-a_{d-1}X^{d-1}-\dots-a_0=\prod_{i=1}^d(X-x_i)$ be the $\mathbb Q_p$-minimal polynomial of $x$, where the $x_i$ are the $\mathbb Q_p$-conjugates of $x$. Since there is a unique extension of the absolute value to $\mathbb Q_p(x)$, we have $|x_1|_p=\dots = |x_d|_p=|x|_p\le1$. Note that the $a_i$ are up to sign given by the elementary symmetric polynomials in the $x_i$. By the ultrametric inequality we therefore have that $|a_i|_p\le 1$ for all $i=0,\dots,d-1$. Thus $x^d=a_0+\dots+a_{d-1}x^{d-1}$ and the $a_i$ are integers.

Assume that for some $n\ge d$ we have $x^n=b_0+\dots+b_{d-1}x^{d-1}$ for some $b_i\in\mathbb Z_p$. Then we have \begin{align*}x^{n+1}&=x(b_0+\dots+b_{d-1}x^{d-1})=b_0x+\dots+b_{d-2}x^{d-1}+b_{d-1}(a_0+\dots+a_{d-1}x^{d-1})\\
&=a_0b_{d-1}+(a_1b_{d-1}+b_0)x+\dots+(a_{d-1}b_{d-1}+b_{d-2})x.\end{align*} Let $\lambda_0=a_0b_{d-1}$ and $\lambda_i=a_{i}b_{d-1}+b_{i-1}, i=1,\dots, d-1$. Since $\mathbb Z_p$ is a ring, we have $\lambda_i\in\mathbb Z_p$ for all $i=0,\dots, d-1$.
\end{proof}
\begin{lemma}\label{lem:bdHQ}Let $Q\in\mathbb Z[X]$ be of degree $m\ge 0$ and $\alpha\in\overline{\mathbb Q}$. Then $h(Q(\alpha))\le \log|Q|_1 + mh(\alpha)$.\end{lemma}
\begin{proof}Let $K=\mathbb Q(\alpha), d=[K:\mathbb Q], d_\nu$ and $|\cdot|_\nu$ as in the proof of Lemma \ref{lem:PatKappa} for all $\nu\in M_K$. Then for every $x\in K$ we have $$h(x)=\frac{1}{d}\sum_{\nu\in M_K} d_\nu \log\max \{1,|x|_\nu\}.$$
If $\nu\in M_K^\infty$, we have $$|Q(\alpha)|_\nu\le \sum_{i=0}^m |q_i|_\nu|\alpha|_\nu^i\le \sum_{i=0}^m |q_i|\max\{1,|\alpha|_\nu\}^m\le |Q|_1\max\{1,|\alpha|_\nu\}^m.$$ For a finite place $\nu\in M_K^0$ we can bound $$|Q(\alpha)|_\nu\le \max_{i=0,\dots, m}\{|q_i|_\nu|\alpha|_\nu^i\}\le \max\{1,|\alpha|_\nu\}^m.$$ 
Using $\sum_{\nu\in M_K^\infty} d_\nu = d$ we conclude that \begin{align*}h(Q(\alpha))&=\frac{1}{d}\sum_{\nu\in M_K} d_\nu \log\max \{1,|Q(\alpha)|_\nu\}\\
&\le\frac{1}{d}\sum_{\nu\in M_K^\infty} d_\nu \log\left(|Q|_1\max\{1,|\alpha|_\nu\}^m\right)+\frac{1}{d}\sum_{\nu\in M_K^0} d_\nu \log(\max \{1,|\alpha|_\nu\}^m)\\
&\le\frac{1}{d}\sum_{\nu\in M_K^\infty} d_\nu \log|Q|_1+\frac{m}{d}\sum_{\nu\in M_K} d_\nu \log\max \{1,|\alpha|_\nu\}\le |Q|_1 + mh(\alpha).\qedhere\end{align*}\end{proof}

\section{Proof of Theorem \ref{thm:BMgen}.}
The proof is the same as the proof of Theorem 1.12 \cite{BM} up to some minor changes. Suppose first that $|\kappa|_p>1$ and let $Q=q_mX^m+\dots + q_0\in\mathbb Z[X]$ with $q_0>0$ such that $Q(\kappa)=0$. Fix a prime $l$ not dividing $q_m$. Let $A_n(X)=X^nQ(X)-\frac{1}{l}$ and let $\alpha_n\in\overline{\mathbb Q}\subseteq\overline{\mathbb Q_p}$ be a root of it. The polynomial $A_n$ has degree $m+n$ and is irreducible by Eisenstein's criterion on $lX^{n+m}A_n(1/X)$. We have $0=A_n(\alpha_n)=\alpha_n^nQ(\alpha_n)-1/l$ and therefore $\alpha_n^{-n}=lQ(\alpha_n)$. Thus, taking the height and using Lemma \ref{lem:bdHQ} and $h(\alpha)=h(\alpha^{-1})$ we find $$nh(\alpha_n)=nh(\alpha_n^{-1})=h(\alpha_n^{-n})=h(lQ(\alpha_n))\le\log|lQ|_1+mh(\alpha_n),$$
which implies that $\lim_{n\to\infty} h(\alpha_n) = 0$. 

Since $Q(\kappa)=0$, we have $A_n(\kappa)=1/l$ and hence $$\frac{1}{n+m}\sum_{\sigma:\mathbb Q(\alpha_n)\to \overline{\mathbb Q_p}} \log|\sigma(\alpha_n)-\kappa|_p = \frac{1}{n+m}\log|A_n(\kappa)/q_m|_p = \frac{-\log|lq_m|_p}{n+m}$$ which tends to $0$ as $n\to \infty.$ As $\log\max\{1,|\kappa|_p\} = \log|\kappa|_p$ we find
$$\lim_{n\to \infty}\left|\frac{1}{[\mathbb Q(\alpha_n):\mathbb Q]}\sum_{\sigma:\mathbb Q(\alpha_n)\to \overline{\mathbb Q_p}} \log|\sigma(\alpha_n)-\kappa|_p - \log\max\{1,|\kappa|_p\}\right|= | \log |\kappa|_p|>0.$$

If $|\kappa|_p<1$, we rename $\kappa$ as $\kappa'$ and use the above calculations for $\kappa = 1/\kappa'$ which satisfies $|\kappa|_p > 1.$ Note that $$\sum_{\mathbb Q(\alpha_n)\to\overline{\mathbb Q_p}} (\log|\sigma(\alpha_n)-\kappa|_p-\log|\sigma(\alpha_n)|_p)= \log|A_n(\kappa)/p_m|_p- \log|A_n(0)/p_m|_p=0$$ since $A_n(\kappa)=A_n(0)=1/l$. For $\alpha_n'=1/\alpha_n$ we have $\sigma(\alpha_n')-\kappa'=\frac{\kappa-\sigma(\alpha_n)}{\sigma(\alpha_n)\kappa}$ and hence 
\begin{align*}\frac{1}{[\mathbb Q(\alpha_n):\mathbb Q]}\sum_{\sigma:\mathbb Q(\alpha_n')\to \overline{\mathbb Q_p}} \log|\sigma(\alpha_n')-\kappa'|_p &= \frac{1}{n+m}\sum_{\mathbb Q(\alpha_n)\to\overline{\mathbb Q_p}} (\log|\sigma(\alpha)-\kappa|_p-\log|\sigma(\alpha_n)|_p-\log|\kappa|_p)\\
&=-\log |\kappa|_p = \log|\kappa'|_p\end{align*} Since $\log\max\{1,|\kappa'|_p\}=0$, we find that the claimed limit holds also in this case.\qed

\section{Proof of Theorems \ref{thm:kappa} and \ref{thm:kappaGen}.}
The proof is organized in three steps. First we construct polynomials $A$ with bounded coefficients and degree which takes a small value at $\kappa$ first in the archemedian and then in the $p$-adic case. Finally, we use this polynomials to construct the sequence we are looking for in a uniform way.

\begin{lemma}\label{lem:arch}Let $\kappa\in\overline{\mathbb Q}\setminus\mu_\infty$ be such that $|\kappa|=1$. Then there exists an increasing sequence $(n_k)_{k\in\mathbb N}$ of natural numbers such that for every $k\in \mathbb N$ there exists $A_k\in\mathbb Z[X]$ of degree at most $n_k$ such that $0<|A_k(\kappa^{n_k})|\le \sqrt 2H(\kappa)^{-(n_k-1)(n_k-\sqrt{n_k})}$ and $|A_k|_1\le 6n_kH(\kappa)^{2n_k}$.\end{lemma}

\begin{proof}Using linearity in rows of the determinant it is easy to deduce the following theorem from Theorem 2C in \cite{Schmidt}: \begin{theorem}\label{thm:MLF} Let $m\in\mathbb N, B\in\mathrm{GL}_m(\mathbb R)$ and $\lambda_1,\dots, \lambda_m>0$ such that $|\det(B)|=\lambda_1\cdots\lambda_m$. Then there exists $\mathbf a\in\mathbb Z^m\setminus\{0\}$ such that $\mathbf b = B\mathbf a$ satisfies $|b_i|< \lambda_i$ for every $1\le i \le m-1$ and $|b_m|\le \lambda_m$.\end{theorem}

Since $\kappa$ is not a root of unity and lies on the unit circle, there are infinitely many $n\in\mathbb N$ such that $\Im(\kappa^n)\ge 1/2$ by Lemma \ref{lem:Imge12}. So let us fix such an $n$. For $0\le i\le n$ let $x_i$ and $y_i$ be the real and imaginary part of $\kappa^{ni}$. In particular, we have $y_1\ge 1/2$. Let us define \begin{align*}B&=\begin{pmatrix}
1 & 0 & \dots &0 & 0 &0\\
0 & \ddots & \ddots & \vdots & \vdots & \vdots\\
\vdots & \ddots& \ddots & 0 & \vdots &\vdots\\
0& \dots& 0 &1 &0 &0\\
y_n &\dots &\dots &y_2 & y_1& 0\\
x_n &\dots & \dots & x_2 & x_1&  1 \end{pmatrix}\in\mathrm{Mat}_{n+1}(\mathbb R),\\
C_{n}&=H(\kappa)^{2(n-\sqrt n)} \text{ and }\epsilon=(y_1C_{n}^{1-n})^{1/2}.\end{align*} Since $\kappa$ is not a root of unity, nor zero we have $H(\kappa)>1$ by Kronecker's theorem. Therefore, $C_n\ge1$ and $\epsilon\le 1$. The determinant of $B$ equals $y_1$. In particular, we have $|\det(B)|=\epsilon^2C_{n}^{n-1}$. Using Theorem \ref{thm:MLF} above with $\lambda_1=\dots=\lambda_{n-1}=C_{n}$ and $\lambda_n=\lambda_{n+1}=\epsilon$ we find a nonzero vector $\mathbf a= (a_n,\dots,a_0)^\top\in\mathbb Z^{n+1}\setminus\{0\}$ and corresponding bounds for $B\mathbf a$. Let $A=a_nX^n+\dots +a_1 X +a_0$ in $\mathbb Z[X]$, then we have $|A(\kappa^{n})|< \sqrt 2\epsilon$ and $|a_i|< C_{n}$ for every $i=2,\dots, n$.

Next we want to show, that $A(\kappa^{n})\ne 0$. The idea is to bound the Mahler measure $M(A)$ from above and below. Note that $|a_1x_1+a_0|=|\Re(A(\kappa^{n}))-a_nx_n-\dots-a_2x_2|\le \epsilon+(n-1)C_{n}\le nC_{n}$ and similarly $|a_1y_1|\le nC_{n}$. Thus, $|a_1|\le nC_{n}/y_1\le 2nC_{n}$ and $|a_0|=|(a_1x_1+a_0)-a_1x_1|\le 3nC_{n}$. Using Lemma \ref{lem:ineqMM} we get the upper bound \begin{equation}\label{eq:upperArch}M(A)\le |A|_1\le (n-1)C_{n}+2nC_{n}+3nC_{n}\le 6nC_{n}.\end{equation}

Assume by contradiction that $A(\kappa^{n})=0$. Let $Q\in\mathbb Z[X]$ be the $\mathbb Z$-minimal polynomial of $\kappa^{n}$. Since $Q$ is irreducible, it divides $A$ in $\mathbb Z[X]$ and therefore we find $M(Q)\le M(A)$. Note that $[\mathbb Q(\kappa^{n}):\mathbb Q]\ge 2$ since otherwise $\kappa^{n}$ is a rational number on the unit circle - which means $\pm1$ - and hence $\kappa$ would be a root of unity. Using properties of the multiplicative height we find $$M(Q)^{1/2}\ge M(Q)^{1/[\mathbb Q(\kappa^{n}):\mathbb Q]}= H(\kappa^{n})=H(\kappa)^{n}.$$ This implies the second lower bound \begin{equation}\label{eq:lowerArch}M(A)\ge M(Q) \ge H(\kappa)^{2n}= H(\kappa)^{2\sqrt n}C_{n}.\end{equation} 
Comparing the bounds \eqref{eq:upperArch} and \eqref{eq:lowerArch} we find $H(\kappa)^{2\sqrt n}\le 6n$, which does not hold if $n$ is large enough. This shows that $A(\kappa^{n})\ne 0$.

Clearly, we have $|A|_1\le 6nC_{n}\le 6nH(\kappa)^{2n}$. We can also bound $$|A(\kappa^n)|\le\sqrt 2\epsilon = \sqrt {2y_1}H(\kappa)^{(1-n)(n-\sqrt n)}.$$ Since this construction works for all $n$ large enough coming from Lemma \ref{lem:Imge12}, we can take them as the sequence.
\end{proof}
We now prove the $p$-adic analogue of Lemma \ref{lem:arch}.

\begin{lemma}\label{lem:padic}Let $p$ be a prime and $\overline{\mathbb Q}\subseteq \overline{\mathbb Q_p}$ an algebraic closure of the rationals. Let $\kappa\in\overline{\mathbb Q}\setminus\mu_\infty$ such that $|\kappa|_p\le 1$. Then for all $n$ large enough in terms of $\kappa$ there exists $A\in\mathbb Z[X]$ of degree at most $n$ such that $0<|A(\kappa^{n})|_p\le pH(\kappa)^{-n(n-\sqrt{n})}$ and $|A|_1\le (n+1)H(\kappa)^{Dn}$, where $D=[\mathbb Q(\kappa):\mathbb Q]$. \end{lemma}

\begin{proof}Let us fix $n\in\mathbb N$, large in terms of $\kappa$ and let $d=[\mathbb Q_p(\kappa^n):\mathbb Q_p]\le D$. We have $|\kappa^n|_p\le1$, thus for every $0\le k\le n$ there are $\lambda_0^{(k)},\dots,\lambda_{d-1}^{(k)}\in\mathbb Z_p$ such that $\kappa^{nk}=\lambda_0^{(k)}+\dots+\lambda_{d-1}^{(k)}\kappa^{d-1}$ by Lemma \ref{lem:intLincomb}.

Let us define $n+1$ linear forms $L^{(k)}(X_0,\dots, X_n)=X_k, k=0,\dots, n$ and $d$ $p$-adic linear forms $L_1^{(i)}=\lambda_i^{(0)}X_0+\dots+\lambda_i^{(n)}X_n, i=0,\dots, d-1$. Moreover, let us define $f=\lfloor n(n-\sqrt n)h(\kappa)/\log(p)\rfloor$ and $\Lambda=p^{df/(n+1)}\le H(\kappa)^{d(n-\sqrt n)}\le H(\kappa)^{Dn}$. Then we have $\Lambda^{n+1}p^{-df}=1$ and hence by Satz 1 in \cite{Mahler} there exists $\mathbf a=(a_0,\dots, a_n)\in\mathbb Z^n\setminus\{0\}$ such that $\max\{|a_0|,\dots, |a_n|\}\le \Lambda$ and $\max\{|L^{(0)}_1(\mathbf a)|_p,\dots,|L^{(d-1)}_1(\mathbf a)|_p\}\le p^{-f}$. Let $A=a_nX^n+\dots+a_0\in\mathbb Z[X]$. Then by construction we have that $A(\kappa^n)=L^{(0)}_1(\mathbf a)+\dots+L^{(d-1)}_1(\mathbf a)\kappa^{d-1}$. In particular, we have that $$|A(\kappa^n)|_p\le p^{-f}\le pH(\kappa)^{-n(n-\sqrt n)}.$$

Next we want to show, that $A(\kappa^{n})\ne 0$. The idea is again to bound the Mahler measure $M(A)$ from above and below. For the upper bound we use \begin{equation}\label{eq:upper}M(A)\le |A|_1\le (n+1)\Lambda\le (n+1)H(\kappa)^{d(n-\sqrt{n})}.\end{equation}

Assume by contradiction that $A(\kappa)=0$. Let $Q\in\mathbb Z[X]$ be the $\mathbb Z$-minimal polynomial of $\kappa^{n}$. Since $Q$ is irreducible, it divides $A$ in $\mathbb Z[X]$ and therefore we find $M(Q)\le M(A)$. Note that $[\mathbb Q(\kappa^{n}):\mathbb Q]\ge d$. Using properties of the multiplicative height we find $$M(Q)^{1/d}\ge M(Q)^{1/[\mathbb Q(\kappa^{n}):\mathbb Q]}= H(\kappa^{n})=H(\kappa)^{n}.$$ This implies the second lower bound \begin{equation}\label{eq:lower}M(A)\ge M(Q) \ge H(\kappa)^{dn}.\end{equation} Comparing the bounds \eqref{eq:upper} and \eqref{eq:lower} we find $H(\kappa)^{\sqrt n}\le H(\kappa)^{d\sqrt n}\le n+1$, which does not hold if $n$ is large enough, since $H(\kappa)>1$ by Kronecker's theorem. This shows that $A(x^{n})\ne 0$.

Recall that $\Lambda\le H(\kappa)^{Dn}$, thus we have $|A|_1\le (n+1)H(\kappa)^{Dn}$. 
\end{proof}

\begin{proof}[Proof of Theorems \ref{thm:kappa} and \ref{thm:kappaGen}]

Let us define $c_\infty=2$ and $c_p=p$ for all primes. Let $\nu\in M_{\mathbb Q}$. Using Lemma \ref{lem:arch} if $\nu=\infty$ or Lemma \ref{lem:padic} with $p=\nu$ we know that there exists an incresing sequence of natural numbers $(n_k)_{k\in\mathbb N}$ such that for every $k\in \mathbb N$ there is a polynomial $A_k\in\mathbb Z[X]$ of degree at most $n_k$ such that $0<|A_k(\kappa^{n_k})|_\nu\le c_{\nu}H(\kappa)^{-(n_k-1)(n_k-\sqrt{n_k})}$ and $|A_k|_1\le 6n_kH(\kappa)^{dn_k}$.

Let us fix some $k$ such that $n_k^2\ge d$ and write $n=n_k$ and $A=A_k$. Let $P\in\mathbb Z[X]$ be the $\mathbb Z$-minimal polynomial of $\kappa$. Let $q$ be the smallest prime which does not divide $P(0)$ and note that this choice does not depend on $n$. Let $\delta\in \{0,1\}$ be one if $q$ divides $A(0)$ and zero otherwise. Then let us consider $$R_n=X^{n^2}P+qA(X^{n})+\delta qP=r_{n^2+d}X^{n^2+d}+\dots + r_0 \in\mathbb Z[X].$$ We have $0\ne |qA(\kappa^{n})|_\nu=|R_n(\kappa)|_\nu\le qc_{\nu}H(\kappa)^{-(n_k-1)(n_k-\sqrt{n_k})}$ and since for large enough $n$ we have $|P|_1\le H(\kappa)^{dn}$ also $|R_n|_1\le |P|_1+q|A|_1+q|P|_1\le 8nqH(\kappa)^{dn}$. The parameter $\delta$ is chosen in such a way that $q$ divides $r_0=qA(0)+\delta qP(0)$ exactly once. It is easy to see that $q$ divides all $r_i, 0\le i\le n^2-1$. Since $q$ does not divide $r_{n^2}=P(0)+qa_n$, we have $R_n\not\equiv 0 \;(\text{mod }p)$. Then by Lemma \ref{lem:Eis}, $R_n$ has a factorization $R_n=S_nT_n, S_n,T_n\in\mathbb Z[X]$, where $T_n$ is irreducible and of degree at least $n^2$. So $\deg(S_n)\le d$.

Next we bound $|T_n(\kappa)|_\nu$ from above. By Lemma \ref{lem:ineqMM} we have \begin{equation}\label{eq:bdL1}|S_n|_1\le 2^{\deg(S_n)}M(S_n)\le 2^dM(S_nT_n)\le 2^d|R_n|_1\le 2^d8nqH(\kappa)^{dn}= 2^{d+3}nqH(\kappa)^{dn}.\end{equation}
Since there is a place $\nu_K$ of $K=\mathbb Q(\kappa)$ such that $|x|_\nu= |x|_{\nu_K}$ for all $x\in K$, we can apply Lemma \ref{lem:PatKappa} to find a positive constant $c_0\le 1$ such that $$|S_n(\kappa)|_\nu\ge c_0^{\deg (S_n)}|S_n|_1^{-d}\ge c_0^d(2^{d+3}qnH(\kappa)^{dn})^{-d}.$$ Since $c_0,d$ and $q$ depend only on $\kappa$, so does $c_1=c_0^d(2^{d+3}q)^{-d}>0$ and we find the bound $|S_n(\kappa)|_p\ge c_1 n^{-d}H(\kappa)^{-d^2n}$. This allows us to bound \begin{align*}|T_n(\kappa)|_\nu&=|R_n(\kappa)|_\nu/|S_n(\kappa)|_\nu\le qc_{\nu}H(\kappa)^{-(n-1)(n-\sqrt{n})}c_1^{-1}n^dH(\kappa)^{d^2n}\\
&\le qc_{\nu}c_1^{-1}n^dH(\kappa)^{d^2n-(n-1)(n-\sqrt n)}\le cn^dH(\kappa)^{(d^2+2)n^{3/2}-n^2}\end{align*} where $c=qc_{\nu}c_1^{-1}$. In particular, we have  $|T_n(\kappa)|_\nu\to 0$ as $n\to \infty$. The Mahler measure is bounded by $M(T_n)\le M(R_n)\le |R_n|_1\le 8nqH(\kappa)^{dn}$.

Now we unfix $k$. We proved so far, that for every $k$ large enough there exists an irreducible $T_k\in\mathbb Z[X]$ of degree $n_k^2\le \deg(T_k)\le n_k^2+d$ such that $|T_k(\kappa)|_\nu\le cn_k^dH(\kappa)^{(d^2+2)n_k^{3/2}-n_k^2}$ and whose Mahler measure satisfies $M(T_k)\le 8n_kqH(\kappa)^{dn_k}$, where $c$ and $q$ do not depend on $k$. Let $\alpha_k\in \overline{\mathbb Q_\nu}$ be a root of $T_k$ and $t_k\in\mathbb Z\setminus\{0\}$ be the leading coefficient of $T_k$, a divisor of the leading coefficient of $P$. In particular, $|t_k|_\nu$ is bounded in terms of $\kappa$. For large enough $k$ we have $|T_k(\kappa)|_\nu\le 1$ and \begin{align*}\limsup_{k\to \infty}\frac{1}{\deg(T_k)}&\sum_{\alpha:T_k(\alpha)=0}\log|\alpha-\kappa|_\nu\le \limsup_{n\to \infty}\frac{\log|T_k(\kappa)|_\nu-\log|t_k|_\nu}{\deg(T_k)}\\
&\le \limsup_{k\to\infty}\frac{\log(c)+d\log(n_k)+((d^2+2)n_k^{3/2}-n_k^2)h(\kappa)}{n_k^2+d}\\
&\le  \limsup_{k\to\infty} -\frac{n_k^2h(\kappa)}{n_k^2+d}=-h(\kappa).\end{align*}

On the other hand we conclude from estimates above that $$h(\alpha_k)= \frac{\log(M(T_k))}{\deg(T_k)}\le\frac{\log(8q)+\log(n_k)+dn_kh(\kappa)}{n_k^2}$$ for all large $k$. As $h(\alpha_k)\ge 0$, this shows that $\lim_{k\to\infty} h(\alpha_k)=0$ and since $\deg(T_k)\ge n_k^2$ and $T_k$ is irreducible, the sequence $(\alpha_k)$ is strict.

Finally, we want to prove a lower bound for $|T_k(\kappa)|_\nu$. Using Lemma \ref{lem:PatKappa} we can bound $|T_k(\kappa)|_\nu\ge H(\kappa)^{-d\deg(T_k)}|T_k|_1^{-d}$. Since $T_k$ divides $R_k$, we can bound $|T_k|_1\le 2^{n_k^2+d+3}n_kqH(\kappa)^{dn_k}$ similarly to \eqref{eq:bdL1}. Then we have \begin{align*}\liminf_{k\to \infty}&\frac{1}{\deg(T_k)}\sum_{\alpha:T_k(\alpha)=0}\log|\alpha-\kappa|_\nu\\&\ge \liminf_{k\to \infty}\frac{\log|T_k(\kappa)|_\nu-\log|t_k|_\nu}{n_k^2}\\
&\ge \liminf_{k\to\infty}\frac{-d(n_k^2+d+dn_k)h(\kappa)-d(n_k^2+d+3)\log(2)-d\log(n_kq)}{n_k^2}\\
&\ge  \liminf_{k\to\infty} \frac{-dn_k^2h(\kappa)-dn_k^2\log(2)}{n_k^2}=-d(h(\kappa)+\log(2)).\qedhere\end{align*}
\end{proof}

\begin{proof}[Proof of Proposition \ref{pr:K}]
For $\kappa=0$, the average is bounded in absolute value by the height in every case and thus $0\in \mathcal{K_\nu}$. For $\kappa=1$, we have Breuillards result or Corollary 1.8 in \cite{BM}. We deduce from them that $\mu_\infty\subseteq \mathcal{K}_\infty$ in Corollary \ref{cor:muInf}. In the $p$-adic case we can deduce that $\mu_\infty\subseteq \mathcal{K}_p$ from Theorem \ref{thm:padicRoots}. For all other $\kappa$ one of the Theorems \ref{thm:BM}, \ref{thm:BMgen}, \ref{thm:kappa} and \ref{thm:kappaGen} implies that $\kappa\not\in\mathcal{K}_\nu$.\end{proof}

\appendix
\section{Roots of unity in the archimedian case}

Let us introduce a modified height, which in \cite{BM} is denoted by $h_d$.
\begin{definition}For an algebraic number $\alpha$ of degree $d=[\mathbb Q(\alpha):\mathbb Q]$ we define $$h'(\alpha)=h(\alpha)+\frac{\log(2d)}{d}>0.$$\end{definition}
\begin{lemma}\label{lem:h'} Let $n\in\mathbb N$ and $\alpha\in\overline{\mathbb Q}$. Let $d=[\mathbb Q(\alpha):\mathbb Q]$. Then we have $h'(\alpha^n)\le nh'(\alpha)$. If $d\ge n^2/2$ we have $h'(\alpha^n)\ge h'(\alpha)/2$.\end{lemma}
\begin{proof}Let $d_n=[\mathbb Q(\alpha^n):\mathbb Q]$. Since $[\mathbb Q(\alpha):\mathbb Q(\alpha^n)]\le n$, we find $d_n\ge d/n$. Thus, we have $$h'(\alpha^n)=h(\alpha^n)+\frac{\log(2d_n)}{d_n}\le nh(\alpha)+\log(2d)\frac{n}{d} = nh'(\alpha).$$
Now suppose that $d\ge n^2/2$. Then we have $\log(2d)\ge2\log(n)$ and thus $\log(n)/\log(2d)\le1/2.$ This implies that $$n\ge 1\ge \frac{\log(2d_n)}{\log(2d)}\ge \frac{\log(2d)-\log(n)}{\log(2d)}\ge 1-1/2 =1/2.$$Hence, we find
$$h'(\alpha^n)=h(\alpha^n)+\frac{\log(2d_n)}{d_n}\ge nh(\alpha)+\frac{\log(2d_n)}{d}\ge \frac{\log(2d_n)}{\log(2d)}\left(h(\alpha)+\frac{\log(2d)}{d}\right)\ge h'(\alpha)/2.\qedhere$$

\end{proof}
\begin{definition}Let $\alpha$ be an algebraic number and $f:D\subseteq\mathbb C\to\mathbb R$ a function which is integrable over the unit circle and defined at the conjugates of $\alpha$. Then as on page $2$ in \cite{BM} we define $$E(f,\alpha)=\left|\frac{1}{[\mathbb Q(\alpha):\mathbb Q]}\sum_{\sigma:\mathbb Q(\alpha)\to\mathbb C} f(\sigma(\alpha))-\int_0^1 f(e^{2\pi i t}) dt\right|,$$ the error between average and integral.\end{definition}
\begin{remark}The limit \eqref{eq:lim} holds if and only if $\lim_{n\to\infty} E(f,\alpha_n)=0$.\end{remark}
\begin{theorem}\label{thm:upper}Let $u\in \mathbb C$ be on the unit circle. Then for every $\alpha\in\overline{\mathbb Q},$ such that $h(\alpha)\le 1$ and no conjugate of $\alpha$ equals $u$ we have the upper bound 
$$\frac{1}{[\mathbb Q(\alpha):\mathbb Q]}\sum_{\sigma:\mathbb Q(\alpha)\to\mathbb C}\log|\sigma(\alpha)-u|\le 64h'(\alpha)^{1/3}\log\left(\frac{4}{h'(\alpha)}\right).$$
\end{theorem}
\begin{proof}The problem with $f_{u}$ is the singularity at $u$. We therefore let $0<\delta\le 1/2$ be a parameter and define $g_\delta(z) = \log \max\{\delta,|z-u|\}\ge\log(\delta)$. As $\delta$ gets smaller, $g_\delta$ approximates $f_u$. By computations very similar to those in the proof of Corollary 1.8 \cite{BM}, one can show that $g_\delta$ satisfies the assumptions of Theorem 1.7 \cite{BM} with the same constants $\tilde M, V,  M_\infty, M_0$ and $\tilde L$. Hence, as in (8.5) \cite{BM} the application of it yields $E(g_\delta,\alpha)\le 28|\log(\delta)|h'(\alpha)^{1/3} +\frac{7}{\delta}h'(\alpha)$.

Let $\theta\in[0, 1)$ be such that $u=e^{2\pi i \theta}$ and define $\mathcal S=\{\sigma: \mathbb Q(\alpha)\to\mathbb C: |\sigma(\alpha)-u|\le \delta\},\\ \Theta=\{t\in [0,1): |e^{2\pi i t}-u|\le \delta\}$ of length $|\Theta|$, 
$$G=\frac{1}{d}\sum_{\sigma\in\mathcal S}\log |\sigma(\alpha)-u|\text{ and } G'=\frac{1}{d}\sum_{\sigma\in\mathcal S}g_\delta(\sigma(\alpha))=\frac{\#\mathcal S}{d}\log(\delta)\le 0,$$ and also $$J=\int_{\Theta}f_u(e^{2\pi i t})dt \text{ and }J'=\int_{\Theta}g_{\delta}(e^{2\pi it})dt=|\Theta|\log(\delta).$$ Note that $G\le G'$. By a change of variables $s=2\pi (t-\theta)$ we find that $J$ and $J'$ are equal to the integrals $J$ and $J'$ in \cite{BM}. Thus, according to (8.10) \cite{BM} we have $|J-J'|\le 2\delta|\log(\delta)|$.

Since $f_u$ and $g_\delta$ do only differ where $|u-z|\le \delta$ and $\int_0^{1} f_u(e^{2\pi it}) dt =0$, we find \begin{align}\frac{1}{d}\sum_{\sigma:\mathbb Q(\alpha)\to\mathbb C}f_u(\sigma(\alpha))&=\frac{1}{d}\sum_{\sigma:\mathbb Q(\alpha)\to\mathbb C}g_\delta(\sigma(\alpha))+G-G'\notag\\
&\le \frac{1}{d}\sum_{\sigma:\mathbb Q(\alpha)\to\mathbb C}g_\delta(\sigma(\alpha))-\int_0^{1}g_{\delta}(e^{2\pi it})dt+\int_0^{1} f_u(e^{2\pi it}) dt + J' -J\notag\\
&\le E(g_\delta,\alpha)+|J-J'|
\notag \\&\le 28|\log(\delta)|h'(\alpha)^{1/3} +\frac{7}{\delta}h'(\alpha)+2\delta|\log(\delta)|.\label{eq:BdDelta}\end{align}

Since $h(\alpha)\le 1$ and $\log(2d)\le d$, we have $h'(\alpha)\le 2$. As in (8.12) \cite{BM} we choose $$\delta=\sqrt{\frac{h'(\alpha)\log(2)}{8\log(4/h'(\alpha))}}\le 1/2$$ and find $|\log(\delta)|\le \frac{26}{25}\log(4/h'(\alpha))$ and thus $$28|\log(\delta)|h'(\alpha)^{1/3}\le 30h'(\alpha)^{1/3}\log(4/h'(\alpha)).$$ Since $h'(\alpha)\le 2$, we have $h'(\alpha)^{1/2}\le 2^{1/6}h'(\alpha)^{1/3}$. Therefore, $$\frac{7}{\delta}h'(\alpha)=7\sqrt 8h'(\alpha)^{1/2}\left(\frac{\log(4/h'(\alpha))}{\log(2)}\right)^{1/2}\le 33h'(\alpha)^{1/3}\log(4/h'(\alpha))\text{ and }$$
$$2\delta|\log(\delta)|\le \frac{52}{25\sqrt 8}h'(\alpha)^{1/2}\log(2)^{1/2}\log(4/h'(\alpha))^{1/2}\le h'(\alpha)^{1/3}\log(4/h'(\alpha)).$$ Plugging in these bounds in \eqref{eq:BdDelta} we find $$\frac{1}{d}\sum_{\sigma:\mathbb Q(\alpha)\to\mathbb C}f_u(\sigma(\alpha))\le 64h'(\alpha)^{1/3}\log\left(\frac{4}{h'(\alpha)}\right).\qedhere$$\end{proof}

\begin{corollary}\label{cor:muInf}Let $\zeta\in\mu_n$. Then for every $\alpha\in \overline{\mathbb Q}\setminus\mu_n$ such that $h(\alpha)\le 1/n$ and $[\mathbb Q(\alpha):\mathbb Q]\ge n^2/2$ we have $$E(f_\zeta,\alpha)\le 104nh'(\alpha)^{1/3}\log\left(\frac{8}{h'(\alpha)}\right).$$ In particular, we have $\mu_\infty\subseteq \mathcal{K}_\infty$.\end{corollary}
\begin{proof}
Let $d=[\mathbb Q(\alpha):\mathbb Q]$. Using that $X^n-1=\prod_{\xi\in\mu_n}(X-\xi)$ we find $$\frac{1}{d}\sum_{\sigma:\mathbb Q(\alpha)\to\mathbb C}\log|\sigma(\alpha^n)-1|=\sum_{\xi\in\mu_n}\frac{1}{d}\sum_{\sigma:\mathbb Q(\alpha)\to\mathbb C}\log|\sigma(\alpha)-\xi|.$$ Lemma \ref{lem:h'} implies that $h'(\alpha^n)^{1/3}\log(4/h'(\alpha^n))\le n^{1/3}h'(\alpha)^{1/3}\log(8/h'(\alpha))$. Thus, with Theorem \ref{thm:upper} and Corollary 1.8 \cite{BM} applied to $\alpha^n$ of height $h(\alpha^n)\le nh(\alpha)\le 1$ we find \begin{align*}-\frac{1}{d}\sum_{\sigma:\mathbb Q(\alpha)\to\mathbb C}\log|\sigma(\alpha)-\zeta|
&=\sum_{\xi\in\mu_n\setminus\{\zeta\}}\frac{1}{d}\sum_{\sigma:\mathbb Q(\alpha)\to\mathbb C}\log|\sigma(\alpha)-\xi|-\frac{1}{d}\sum_{\sigma:\mathbb Q(\alpha)\to\mathbb C}\log|\sigma(\alpha^n)-1|\\
&\le (n-1)64h'(\alpha)^{1/3}\log(4/h'(\alpha)))+40h'(\alpha^n)^{1/3}\log\left({4}/{h'(\alpha^n)}\right)\\
&\le 64nh'(\alpha)^{1/3}\log(8/h'(\alpha))+40n^{1/3}h'(\alpha)^{1/3}\log\left({8}/{h'(\alpha)}\right)\\
&\le 104nh'(\alpha)^{1/3}\log(8/h'(\alpha)).\end{align*}
Together with the bound from Theorem \ref{thm:upper} and since the integral $\int_0^1 f_\zeta(e^{2\pi i t}) dt$ vanishes we have $$E(f_\zeta,\alpha)=\left|\frac{1}{d}\sum_{\sigma:\mathbb Q(\alpha)\to\mathbb C}\log|\sigma(\alpha)-\zeta|\right|\le 104nh'(\alpha)^{1/3}\log\left(\frac{8}{h'(\alpha)}\right).$$

Now let $\zeta\in\mu_\infty$ be of order $N$. Let $(\alpha_n)_{n\in\mathbb N}$ be a strict sequence of algebraic numbers such that $\alpha_n$ is never a conjugate to $\zeta$ and such that $\lim_{n\to\infty} h(\alpha_n)=0$. Since the sequence is strict, there exists $n_0 \in\mathbb N$ such that for all $n\ge n_0$ we have $\alpha_n\notin \mu_N$, $h(\alpha_n)\le 1/N$ and $[\mathbb Q(\alpha_n):\mathbb Q]\ge N^2/2$ by the Northcott property. Hence, it satisfies the assumptions of the first part of this corollary. Strictness and the Northcott property imply that also $\lim_{n\to\infty}h'(\alpha_n)=0$. Thus, we have $$\lim_{n\to\infty}E(f_\zeta,\alpha_n)\le \lim_{n\to\infty}104Nh'(\alpha_n)^{1/3}\log(8/h'(\alpha_n))=0$$ and hence $\zeta\in \mathcal K_\infty$.\end{proof}

\section{Roots of unity in the $p$-adic case}
In this section we work with $\overline{\mathbb Q}\subseteq \overline{\mathbb Q_p}$ for some prime $p$.
\begin{lemma}\label{lem:1}Let $\alpha\in\overline{\mathbb Q}\setminus\{1\}$ such that $h(\alpha)\le 1$. Then we have $$\left|\frac{1}{[\mathbb Q(\alpha):\mathbb Q]}\sum_{\sigma:\mathbb Q(\alpha)\to \overline{\mathbb Q}} \log|\sigma(\alpha)-1|_p\right| \le 40h'(\alpha)^{1/3}\log\left(\frac{4}{h'(\alpha)}\right)+h(\alpha).$$\end{lemma}
\begin{proof}Let $P=a_dX^d+\dots+a_0\in\mathbb Z[X]$ be the $\mathbb Z$-minimal polynomial of $\alpha$. The average is given by $$A:=\frac{1}{d}\sum_{\sigma:\mathbb Q(\alpha)\to \overline{\mathbb Q}} \log|\sigma(\alpha)-1|_p=\frac{\log|P(1)|_p-\log|a_d|_p}{d}$$ and we now prove upper and lower bounds.

 Since $P(1)$ is an integer, we have $\log|P(1)|_p\le 0$. Note that for an integer $k\in\mathbb Z\setminus \{0\}$ we have $|k|_p\ge |k|^{-1}$. This implies that $-\log|a_d|_p\le \log|a_d|$. Therefore, we find the upper bound $A\le \frac{\log|a_d|}{d}\le h(\alpha)$.
 
For the lower bound we have $\log|a_d|_p\le 0$ and $-\log|P(1)|_p\le \log|P(1)|$. R. Baker and D. Masser proved in Corollary 1.8 \cite{BM} that $\frac{\log|P(1)|-\log|a_d|}{d}\le 40h'(\alpha)^{1/3}\log\left({4}/{h'(\alpha)}\right)$.
We therefore find $$-A\le \frac{\log|P(1)|}{d}=\frac{\log|P(1)|-\log|a_d|+\log|a_d|}{d}\le 40h'(\alpha)^{1/3}\log\left({4}/{h'(\alpha)}\right)+h(\alpha).\qedhere$$\end{proof}

\begin{lemma}\label{lem:upper}Let $x\in \overline{\mathbb Q}$ be such that $|x|_p\le1$. Then for every $\alpha\in\overline{\mathbb Q},$ such that $h(\alpha)\le 1$ and no conjugate of $\alpha$ equals $x$ we have the upper bound 
$$\frac{1}{[\mathbb Q(\alpha):\mathbb Q]}\sum_{\sigma:\mathbb Q(\alpha)\to\overline{\mathbb Q}}\log|\sigma(\alpha)-x|\le h(\alpha).$$\end{lemma}
\begin{proof}Let $P=a_dX^d+\dots+a_0\in\mathbb Z[X]$ be the $\mathbb Z$-minimal polynomial of $\alpha$. The average is given by $$A:=\frac{1}{d}\sum_{\sigma:\mathbb Q(\alpha)\to \overline{\mathbb Q}} \log|\sigma(\alpha)-x|_p=\frac{\log|P(x)|_p-\log|a_d|_p}{d}.$$

 Since $|x|_p\le 1$, we have $\log|P(x)|_p\le 0$. Note that for an integer $k\in\mathbb Z\setminus \{0\}$ we have $|k|_p\ge |k|^{-1}$. This implies that $-\log|a_d|_p\le \log|a_d|$. Therefore, we find the upper bound $A\le \frac{\log|a_d|}{d}\le h(\alpha)$.\end{proof}

\begin{theorem}\label{thm:padicRoots}Let $\zeta\in\mu_\infty$ be a root of unity of order $n$. Then for every $\alpha\in\overline{\mathbb Q}\setminus\mu_n$ such that $h(\alpha)\le 1/n$ we have $$\left|\frac{1}{d}\sum_{\sigma:\mathbb Q(\alpha)\to \overline{\mathbb Q}} \log|\sigma(\alpha)-\zeta|_p\right| \le 40h'(\alpha^n)^{1/3}\log\left(\frac{4}{h'(\alpha^n)}\right)+2h(\alpha^n).$$\end{theorem}

\begin{proof}Let $d=[\mathbb Q(\alpha):\mathbb Q]$. Using that $X^n-1=\prod_{\xi\in\mu_n}(X-\xi)$ we find $$\frac{1}{d}\sum_{\sigma:\mathbb Q(\alpha)\to\overline{\mathbb Q}}\log|\sigma(\alpha^n)-1|_p=\sum_{\xi\in\mu_n}\frac{1}{d}\sum_{\sigma:\mathbb Q(\alpha)\to\overline{\mathbb Q}}\log|\sigma(\alpha)-\xi|_p.$$ Thus, with Lemma \ref{lem:1} applied to $\alpha^n\ne 1$ of height $h(\alpha^n)=nh(\alpha)\le 1$ and Lemma \ref{lem:upper} we find \begin{align*}-\frac{1}{d}\sum_{\sigma:\mathbb Q(\alpha)\to\overline{\mathbb Q}}\log|\sigma(\alpha)-\zeta|_p
&=\sum_{\xi\in\mu_n\setminus\{\zeta\}}\frac{1}{d}\sum_{\sigma:\mathbb Q(\alpha)\to\overline{\mathbb Q}}\log|\sigma(\alpha)-\xi|_p-\frac{1}{d}\sum_{\sigma:\mathbb Q(\alpha)\to\overline{\mathbb Q}}\log|\sigma(\alpha^n)-1|_p\\
&\le (n-1)h(\alpha)+40h'(\alpha^n)^{1/3}\log\left({4}/{h'(\alpha^n)}\right)+h(\alpha^n)\\
&\le 40h'(\alpha^n)^{1/3}\log\left(\frac{4}{h'(\alpha^n)}\right)+2h(\alpha^n).\qedhere\end{align*}\end{proof}


\begin{thebibliography}{[1]}
\bibitem{Aut} P. Autissier, \textit{Sur une question d’équirépartition de nombres algébriques}, C. R. Math. Acad. Sci. Paris, vol 342, no 9, pp. 639-641, 2006
\bibitem{BM} R. Baker and D. Masser, \textit{Galois Distribution on Tori - A Refinement, Examples, and Applications}, International Mathematics Research Notices, pp. 1-65, 2022
\bibitem{Bilu} Y. Bilu, \textit{Limit distribution of small points on algebraic tori}, Duke Math. J. 89, pp. 465–476, 1997
\bibitem{BG} E. Bombieri and W. Gubler, \textit{Heights in Diophantine Geometry}, Cambridge University Press, 2006.
\bibitem{Br} E. Breuillard, \textit{A height gap theorem for finite subsets of $\mathrm{GL}_d(\mathbb Q)$ and nonamenable subgroups}, Ann. of Math. vol 174, no 2, pp. 1057–1110, 2011
\bibitem{CLT} A. Chambert-Loir and A. Thuillier, \textit{Mesures de Mahler et équidistribution logarithmique}, Ann. Inst. Fourier vol 59, no 3 pp. 977-1014, 2009
\bibitem{Fr} L. Frey, \textit{Explicit Small Heights in Infinite Non-Abelian Extensions}, Acta Arith. 199, pp. 111-133, 2021
\bibitem{Mahler} K. Mahler, \textit{Über Diophantische Approximationen in Gebiete der $p$-adischen Zahlen}, Jahresbericht d. Deutschen Math. Verein. vol. 44, pp. 250-255, 1934
\bibitem{PST} J. Pineiro, L. Szpiro, T.J. Tucker, \textit{Mahler measure for dynamical systems on $\mathbb{P}^1$ and intersection theory on a singular arithmetic surface} In: Geometric Methods in Algebra and Number Theory, pp. 219–250, Prog.
Math. 235, Birkh\"auser, 2005
\bibitem{Schmidt} W. Schmidt, \textit{Simultaneous Approximation.} In: Diophantine Approximation. Lecture Notes in Mathematics, vol 785, pp. 27-47, Springer, Berlin, Heidelberg, 1980
\end{thebibliography}
\end{document}